\documentclass[reqno,12pt]{amsart}

\usepackage{amsfonts,amsmath,amsthm,amssymb,amscd}
\usepackage[all]{xy}

\usepackage{color}

\usepackage[mathscr]{eucal}
\usepackage{mathrsfs}
\usepackage[final]{showkeys}

\newcommand{\version}{Ver.~0.0}
\newcommand{\setversion}[1]{\renewcommand{\version}{Ver.~{#1}}}
\setversion{0.0 [2018/06/11 12:14:57 CEST]}
\setversion{1.0 [2018/10/10 15:25:56 EDT]}
\setversion{1.1 [2018/10/13 19:30:10 EDT]}
\setversion{1.2 [2018/11/14]}
\setversion{1.3 [2018/11/18]}
\setversion{1.31 (almost complete) [2018/11/23 11:33:45 EST]}
\setversion{1.32 [2018/11/26 11:42:02 EST]}
\setversion{2.0 [2019/05/05 12:43:45 JST]}

\theoremstyle{plain}
\newtheorem{theorem}{Theorem}
\newtheorem{proposition}[theorem]{Proposition}
\newtheorem{corollary}[theorem]{Corollary}
\newtheorem{lemma}[theorem]{Lemma}

\theoremstyle{definition}

\theoremstyle{remark}
\newtheorem{remark}[theorem]{\upshape Remark}

\numberwithin{equation}{section}
\numberwithin{theorem}{section}

\newcommand{\Z}{\mathbb{Z}}

\newcommand{\R}{\mathbb{R}}

\newcommand{\C}{\mathbb{C}}

\newcounter{thmenum}

\newenvironment{mynote}{\par\medskip\noindent\begin{math} \blacktriangleright \end{math}{\bfseries [Memorandom]}\ }
{\begin{math} \blacktriangleleft \end{math}\par\medskip}

\newcommand{\vectwo}[2]{{\renewcommand{\arraystretch}{.85}\Bigl(\begin{array}{@{\,}c@{\,}}{#1}\\ {#2}\end{array}\Bigr)}}

\newcommand{\mattwo}[4]{\Bigl(\begin{array}{@{\,}c@{\;\;}c@{\,}}{#1} & {#2} \\ {#3} & {#4} \end{array}\Bigr)}

\newlength{\lengthcup}
\newcommand{\trace}{\qopname\relax o{Tr}}

\newcommand{\sgn}{\qopname\relax o{sgn}}

\renewcommand{\Re}{\qopname\relax o{Re}}
\renewcommand{\Im}{\qopname\relax o{Im}}

\newcommand{\Image}{\qopname\relax o{Im}}

\newcommand{\transpose}[1]{\,{}^t{#1}\,}

\newcommand{\restrict}{\big|}

\newcommand{\Sym}{\mathop{\mathrm{Sym}}\nolimits}

\newcommand{\Str}{\mathop\mathrm{Str}\nolimits{}}

\newcommand{\calorbit}{\mathcal{O}}

\newcommand{\GL}{\mathrm{GL}}

\newcommand{\OO}{\mathrm{O}}

\newcommand{\Mat}{\mathrm{M}}
\newcommand{\regMat}{\mathrm{M}^{\circ}}

\newcommand{\la}{\langle}
\newcommand{\ra}{\rangle}

\newcommand{\GC}{G_{\C}}

\newcommand{\VV}{V}  
\newcommand{\EE}{E}  
\newcommand{\WW}{W}  
\newcommand{\LL}{L}  
\newcommand{\HH}{H}  
\newcommand{\GG}{G}  

\newcommand{\tildeOmega}{\widetilde{\Omega}}
\newcommand{\ztilde}{\widetilde{z}}
\newcommand{\wtilde}{\widetilde{w}}

\newcommand{\SchwarzSpace}{\mathscr{S}}
\newcommand{\Zetapcone}{Z_{\tildeOmega}{}}
\newcommand{\Kernelpcone}{K^+}
\newcommand{\Kernelrho}{K^{\rho}}
\newcommand{\delzy}{\partial_{z, y}}

\newcommand{\II}{\mathbb{I}}

\newcommand{\skipover}[1]{}

\title{Enhanced zeta distributions and its functional equations}

\author{Kyo Nishiyama}
\address{Department of Mathematics, 
Aoyama Gakuin University,
Fuchinobe 5-10-1, Sagamihara 252-5258, Japan}
\email{kyo@gem.aoyama.ac.jp}
\author{Bent {\O}rsted}
\address{Department of Mathematics, Aarhus University,
Ny Munkegade, 8000 Aarhus C, Denmark}
\email{orsted@math.au.dk}
\author{Akihito Wachi}
\address{Department of Mathematics, 
Hokkaido University of Education, 
Shiroyama 1, Kushiro 085-8580, Japan}
\email{wachi@kus.hokkyodai.ac.jp}

\begin{document}

\maketitle

\begin{abstract}
We consider an ``enhanced symmetric space'', which is a prehomogeneous vector space.  
This vector space is intimately related to a double flag variety studied in \cite{NO.2018}.  
On a distinguished open orbit called ``enhanced positive cone'', 
we consider a zeta integral with two complex variables, 
which is analytically continued to meromorphic family of tempered distributions.  
One of the main results of this paper is to establish a precise formula for the meromorphic continuation which clarifies the location of poles (and may be useful to obtain residues).  
We also compute the Fourier transform of the zeta distribution and obtain 
a functional equation with explicit gamma factors.
\end{abstract}




\section{Introduction}

Zeta distributions or zeta integrals are studied extensively by many authors.  
Most classical one is the Tate's zeta integral:
for a Schwarz function $ \varphi \in \SchwarzSpace(\R) $, 
\begin{equation*}
Z^{\mathrm{T}}(\varphi, s) = \int_{\R} \varphi(z) |z|^s dz \qquad (s \in \C).
\end{equation*}
The integral converges for $ \Re s > -1 $ and admits a meromorphic continuation to all the complex plane.  A generalization is the Godement-Jacquet zeta integral:
\begin{equation*}
Z^{\mathrm{GJ}}(\varphi, s) = \int_{\Mat_n(\R)} \varphi(z) |{\det z}|^s dz ,
\end{equation*}
with obvious notations.  
Also we can consider zeta integrals over the space of symmetric matrices $ \Sym_n(\R) $.  
In this case, we will consider the integral over the set of matrices with a fixed signature: 
$ \Omega(p, q) = \{ z \in \Sym_n(\R) \mid \sgn(z) = (p, q) \} \; (p + q = n) $.  
These zeta integrals are important in number theory 
since they produce zeta functions for lattices.

Around 1970's, by T.~Shintani and M.~Sato (\cite{Sato.Shintani.1974}), 
these zeta integrals are generalized to those for prehomogeneous vector spaces. 
Let $ \GC $ be a reductive algebraic group and assume 
that it linearly acts on a vector space $ V_{\C} $ \emph{with an open orbit} $ \calorbit $.  
In other words, the pair $ (\GC, V_{\C}) $ is a \emph{prehomogeneous vector space}, 
which is abbreviated to PV.  
For the properties of PV, readers may refer to \cite{Kimura.PV.book} or \cite{Sato.Shintani.1974, Sato.M.1970.PV}.  
To proceed further, 
we assume that 
$ S = V_{\C} \setminus \calorbit $ is an irreducible hypersurface which implies 
the existence of a unique fundamental relative invariant $ P(z) \in \C[V_{\C}] $ and 
$ S $ is defined by $ P(v) = 0 $.  

Let $ V_{\R} $ be a real form of $ V_{\C} $ on which 
$ G_{\R} $ acts.  
Then $ V_{\R} \setminus \{ z \in V_{\R} \mid P(z) = 0 \} $ breaks up into several 
open $ G_{\R} $ orbits:
\begin{equation*}
V_{\R} \setminus \{ z \in V_{\R} \mid P(z) = 0 \} = 
\textstyle\bigcup\nolimits_{i = 1}^{\ell} \calorbit_i .
\end{equation*}
On each $ \calorbit_i $, we define a zeta integral 
\begin{equation*}
Z_i^{(G,V)}(\varphi, s) = \int_{\calorbit_i} \varphi(z) |P(z)|^s dz 
\qquad
(\varphi \in \SchwarzSpace(V_{\R}), \; s \in \C), 
\end{equation*}
regarded as a tempered distribution.  
For these zeta integrals, 
Sato and Shintani proved a \emph{fundamental theorem} which states
(1) existence of a meromorphic continuation of the distribution in the parameter $ s \in \C $; 
(2) duality with respect to the Fourier transform, which is called \emph{functional equations}.  
This theorem is very powerful and investigated by many authors, but in practice, 
it needs a serious work to determine the meromorphic continuation 
(e.g., to locate poles or to obtain residues) and 
to deduce explicit functional equations.
It is out of the authors' scope to list all the relevant works which 
study explicit forms of zeta integrals and their properties.  
However, among many of them, 
we would like to mention 
\cite{{Shintani.1975},
{Suzuki.1979},
{Suzuki.1975.kokyuroku}, 
{Satake.Faraut.1984},
{Muro.1986},
{Clerc.2002}, 
{Barchini.2004}, 
{Barchini.Sepanski.Zierau.2006}, 
{Sato.2006.JMSJ},
{Sato.2006.StPauli},
{BenSaid.Clerc.Koufany.2017arXiv170401817B}} 
and references therein.

Later, Fumihiro Sato generalized the fundamental theorem to several complex variables 
(\cite{Sato.PVzeta.I.1982, Sato.PVzeta.II.1983, Sato.PVzeta.III.1982}).  

In this note, we will study the above two problems for 
a prehomogeneous vector space called an ``enhanced symmetric space'' 
associated with a double flag variety (see \cite{NO.2018}).  
In this case, there are two fundamental relative invariants so that we consider 
a zeta integral in two complex variables $ s = (s_1, s_2) $.  
We get an explicit meromorphic continuation (or gamma factors) indicating poles 
(Theorem~\ref{thm:analytic.continuation})
and also explicit functional equations in the case of ``enhanced positive cone''
(Theorems~\ref{thm:FT.Kernelp} and \ref{thm:FEq.on.open.orbits}).  
To state the results more precisely, we need to prepare the notations and settings.

\section{Enhanced zeta integral and its meromorphic continuation}\label{sec:meromorphic.continuation}

Let $ \VV = \Sym_n(\R) $ be the space of real symmetric matrices.  
Then $ \VV $ inherits a structure of Euclidean Jordan algebra (see \cite{Faraut.Koranyi.1994}).  
Let us denote $ \EE = \Mat_{n,d}(\R) $, a representation space of $ \VV $, 
where $ z \in \VV $ acts on $ \EE $ by the left multiplication and as a consequence it admits 
a natural action of $ \LL = \Str(V) = \GL_n(\R) $, the structure group of the Jordan algebra $ V $.  
On the other hand, 
$ \EE $ can be regarded as a direct sum of $ d $-copies of the natural representation 
$ \R^n $ of $ \VV $.  
This allows an action of $ \HH = \GL_d(\R) $ on $ \EE $ by 
the right multiplication.  

Consider a direct sum $ \WW = \VV \oplus \EE = \Sym_n(\R) \oplus \Mat_{n,d}(\R) $ 
and let us denote $ \GG = \LL \times \HH = \GL_n(\R) \times \GL_d(\R) $.  
There is a natural action of $ \GG $ on $ \WW $ via 
\begin{equation*}
(g, h) \cdot (z, y) = (g z \transpose{g}, g y \transpose{h}) , 
\quad \text{ where } 
(g, h) \in \LL \times \HH = \GG , \;\; 
(z, y) \in \VV \oplus \EE = \WW .
\end{equation*}
If we extend the base field $ \R $ to $ \C $, 
the space $ \WW_{\C} = \Sym_n(\C) \oplus \Mat_{n,d}(\C) $ is a prehomogeneous vector space with the action of $ \GG_{\C} $ and it has only finitely many $ \GG_{\C} $ orbits.

In the theory of prehomogeneous vector spaces, 
relative invariants and their $ b $-functions play important roles.  
If $ d \leq n $, the fundamental relative invariants in $ \C[\WW_{\C}] $ 
(the space of regular functions, or polynomials over $ \WW_{\C} $) are given by 
\begin{equation*}
P_1(z, y) = \det z, \qquad
P_2(z, y) = (-1)^d \det \mattwo{z}{y}{\!\!\transpose{y}}{0}
\qquad
((z, y) \in \WW_{\C}),
\end{equation*}
where $ P_1 $ (respectively $ P_2 $) is associated with the character 
$ \chi_{P_1}(g, h) = (\det g)^2 $ 
(respectively $ \chi_{P_2}(g, h) = (\det g)^2 (\det h)^2 $) 
for $ (g, h) \in \GG_{\C} $.  
Note that if $ z \in \Sym_n(\C) $ is regular we can rewrite 
\begin{equation*}
P_2(z, y) = \det z \cdot \det \bigl(\transpose{y} z^{-1} y \bigr) , 
\end{equation*}
and this is the reason why we put $ (-1)^d $ in the definition of $ P_2 $.  
If $ z $ is a positive definite (real) symmetric matrix and $ y $ is also a real matrix, 
$ P_i(z, y) \; (i = 1, 2) $ are both positive.   
Any relative invariants are of the form $ P_1^{m_1} P_2^{m_2} $ with 
$ m_1, m_2 \in \Z_{\geq 0} $.  

If $ d > n $, the relative invariant $ P_2(z, y) $ above vanishes identically, and only $ P_1 $ survives.  In fact, in this case, any relative invariants are some powers of $ P_1 $.  
Thus, throughout this paper, 
\emph{we always assume the inequality $ d \leq n $}.

For the fundamental relative invariants $ P_1 $ and $ P_2 $, 
let us describe $ b $-functions in two variables.  
We denote by $ P_i^{\ast}(\delzy) $ a constant coefficient differential operator 
which is defined by 
\begin{equation*}
P_i^{\ast}(\delzy) e^{\trace z w + \trace \transpose{y} x} 
= P_i(w, x) e^{\trace z w + \trace \transpose{y} x} 
\qquad
(z, w \in V, y, x \in E, i = 1, 2).
\end{equation*}

\begin{proposition}[{\cite[Propositions~6 \& 7]{Wachi.2017.kokyuroku}}]\label{introprop:b-function}
Let $ s = (s_1, s_2) \in \C^2 $ be a pair of complex variables.  
We put 
\begin{align*}
b_{1,0}(s) 
&= \prod_{j = 1}^d (s_1 + \dfrac{d + 1}{2} - \dfrac{j - 1}{2}) 
\prod_{k = 1}^{n - d} (s _1 + s_2 + \dfrac{n + 1}{2} - \dfrac{k - 1}{2}), 
\\
b_{0,1}(s) 
&= \prod_{j = 1}^d (s_2 + \dfrac{d + 1}{2} - \dfrac{j - 1}{2}) (s_2 + \dfrac{n}{2} - \dfrac{j - 1}{2}) 
\prod_{k = 1}^{n - d} (s _1 + s_2 + \dfrac{n + 1}{2} - \dfrac{k - 1}{2}) .
\end{align*}
Then the following Bernstein-Sato identity holds:
\begin{align*}
P_1^{\ast}(\delzy) \Bigl( P_1(z, y)^{s_1 + 1} P_2(z, y)^{s_2} \Bigr)
&= b_{1,0}(s_1, s_2) P_1(z, y)^{s_1} P_2(z, y)^{s_2} , 
\\
P_2^{\ast}(\delzy) \Bigl( P_1(z, y)^{s_1} P_2(z, y)^{s_2 + 1} \Bigr)
&= b_{0,1}(s_1, s_2) P_1(z, y)^{s_1} P_2(z, y)^{s_2} .
\end{align*}
The polynomials $ b_{1,0}(s) $ and $ b_{0,1}(s) $ are called $ b $-functions.
\end{proposition}

Let us return back to the base field $ \R $.  
There is a unique open orbit over $ \C $, but it breaks up into several open orbits over $ \R $.  
It is not difficult to parametrize all of them (see Lemma~\ref{lemma:open.orbits} below), 
and among them, there is a distinguished open orbit denoted as
\begin{equation*}
\tildeOmega = \Omega \times \regMat_{n,d}(\R) , 
\end{equation*}
where 
$ \Omega = \Sym_n^+(\R) $ is the space of real positive symmetric matrices (a homogeneous symmetric cone contained in  $ V $ in the terminology of \cite{Faraut.Koranyi.1994}) and 
$ \regMat_{n,d}(\R) $ is the set of full rank matrices in $ E = \Mat_{n,d}(\R) $.  
We call $ \tildeOmega $ an \emph{enhanced positive symmetric cone}.

We are interested in the following integral on $ \tildeOmega $, 
which is called an \emph{enhanced zeta integral} (or \emph{distribution}).  
\begin{equation*}
\begin{aligned}
\Zetapcone(\varphi, s) 
&= \int_{\tildeOmega} \varphi(z, y) P_1(z, y)^{s_1} P_2(z, y)^{s_2} dz dy
\\
&= \int_{\Sym_n^+(\R)} (\det z)^{s_1} \, dz 
\int_{\Mat_{n,d}(\R)} \varphi(z, y) \Bigl| \det \mattwo{z}{y}{\transpose{y}}{0} \Bigr|^{s_2} \, dy
\end{aligned}
\end{equation*}
Here $ s = (s_1, s_2) $ are complex parameters and $ \varphi(z, y) \in \SchwarzSpace(\WW) $ is 
a rapidly decreasing function (Schwarz function).  The measures 
$ dz $ and $ dy $ are (restriction of) Lebesgue measures.  
The integral converges if $ \Re s_1, \Re s_2 \gg 0 $ are sufficiently large, 
and for the rest of $ s $, we will consider analytic continuation as a tempered distribution.  
For this, we need some more notation.
For a non-negative integer $ k \geq 0 $ and $ \alpha \in \C $, we put 
\begin{equation}\label{eqn:multiple.Gamma}
\Gamma_k(\alpha) = \Gamma(\alpha) \, \Gamma(\alpha - \dfrac{1}{2}) \cdots \Gamma(\alpha - \dfrac{k - 1}{2}) 
= \prod\nolimits_{j = 1}^k \Gamma(\alpha - \dfrac{j - 1}{2}) , 
\end{equation}
and for our enhanced positive cone $ \tildeOmega $, define 
\begin{equation}\label{eqn:Gamma.factor.EPC}
\Gamma_{\tildeOmega}(s) = \Gamma_d(s_1 + \dfrac{d + 1}{2}) \, 
\Gamma_d(s_2 + \dfrac{d + 1}{2}) \, \Gamma_d(s_2 + \dfrac{n}{2}) \, 
\Gamma_{n - d}(s_1 + s_2 + \dfrac{n + 1}{2}) .
\end{equation}
Then the standard arguments of the meromorphic continuation using $ b $-functions 
together with Proposition~\ref{introprop:b-function} 
prove the following theorem.  

\begin{theorem}[Meromorphic Continuation]\label{thm:analytic.continuation}
The zeta integral normalized by the gamma factor 
\begin{equation*}
\dfrac{1}{\;\Gamma_{\tildeOmega}(s)\;} \, \Zetapcone(\varphi, s) 
\end{equation*}
is extended to an entire function in $ s = (s_1, s_2) \in \C^2 $ 
for any $ \varphi \in \SchwarzSpace(\WW) $.  
In other words, $ \Zetapcone(\varphi, s) $ can be extended to a meromorphic function 
with possible poles specified by the gamma factor $ \Gamma_{\tildeOmega}(s) $.
\end{theorem}

\section{Fourier transform and a functional equation}\label{section:FT.main.results}

Let $ \calorbit_i \; (i = 1, \dots, \ell) $ be open $ \GG $ orbits in $ \WW $.  
The enhanced positive cone $ \tildeOmega $ is one of them.  

\begin{lemma}[\cite{NO.2018}]\label{lemma:open.orbits}
A $ G $-orbit through $ (z, y) \in W $ is open 
if and only if 
$ z \in \Sym_n(\R) $ 
and $ \transpose{y} z^{-1} y \in \Sym_d(\R) $ are both regular.  
Open orbits are classified by 
the signature of $ z $ and the signature of 
$ \transpose{y} z^{-1} y $.  
\end{lemma}

\begin{remark}
Let us denote the image of the matrix $ y $ by $ [y] := \Im y \subset \R^n $, 
which is a $ d $-dimensional subspace of $ \R^n $.  
Then the signature of $ \transpose{y} z^{-1} y $ is equal to that of the quadratic form 
$ z^{-1} \restrict_{[y]} $ restricted to $ [y] $.  
Note that there is a natural condition 
on the signature which is induced from $ z^{-1} $ (or $ z $).  
See \S~\ref{subsec:open.orbits} for details.
\end{remark}

Our zeta integral $ \Zetapcone(\varphi, s) $ is considered to be 
a tempered distribution supported in the closure of $ \tildeOmega $.  
Let us calculate the Fourier transform of the distribution.  

Write $ \ztilde = (z, y) $ and we denote the integral kernel by
\begin{equation}\label{eqn:Kernelpcone}
\Kernelpcone_s(\ztilde) 
= \begin{cases}
(\det z)^{s_1} \, \bigl| \det \mattwo{z}{y}{\transpose{y}}{0} \bigr|^{s_2} 
& \ztilde \in \tildeOmega 
\\
0 
&
\text{otherwise.}
\end{cases}
\end{equation}
We define a positive definite inner product on $ \WW = \Sym_n(\R) \oplus \Mat_{n, d}(\R) $ by 
\begin{equation}\label{eqn:inner.product}
\la \ztilde, \wtilde \ra = 
\trace z w + \trace \transpose{y} x 
\qquad
\text{ for } 
\quad
\ztilde = (z, y), \, \wtilde = (w, x) \in W, 
\end{equation}
and the Euclidean Fourier transform $ \mathscr{F} \varphi = \widehat{\varphi} $ is defined 
as usual:
\begin{equation*}
\widehat{\varphi}(\wtilde) 
= \int_{\WW} \varphi(\ztilde) e^{- 2 \pi i \la \ztilde, \wtilde \ra} d\ztilde.  
\end{equation*}
Note that the Fourier inversion becomes 
\begin{equation*}
(\mathscr{F}^{-1}\psi) (\ztilde) 
= 2^{- \frac{n(n -1)}{2}}\int_{\WW} \psi(\wtilde) e^{2 \pi i \la \wtilde, \ztilde \ra} d\wtilde, 
\end{equation*}
because of the normalization of the inner product.  

To state the result, we need a definition of the boundary value distribution 
of an analytic function, i.e., a hyperfunction.  We do not need a general theory, so we just put 
\begin{equation}\label{eqn:hyperfunction.Xi}
\begin{aligned}
\Xi_s(\wtilde) 
&= P_1( +0 + 2 \pi i w, x)^{s_1} P_2( +0 + 2 \pi i w, x)^{s_2}
\\
&= \lim_{v \downarrow 0} \det(v + 2 \pi i w)^{s_1} 
\Bigl( (-1)^d \det \mattwo{v + 2 \pi i w}{x}{\transpose{x}}{0} \Bigr)^{s_2},
\end{aligned}
\end{equation}
where $ v \in \Omega = \Sym_n^+(\R) $ moves to $ 0 $ in the positive cone.  
This is a well defined distribution supported in the whole $ \WW $ interpreted as 
\begin{equation*}  
\int_{\WW} 
\Xi_s(\wtilde) \varphi(\wtilde) d\wtilde
= \lim_{v \downarrow 0} 
\int_{\WW} P_1(v + 2 \pi i w, x)^{s_1} P_2(v + 2 \pi i w, x)^{s_2} 
\varphi(\wtilde) d\wtilde 
\end{equation*}
for a test function $ \varphi \in \SchwarzSpace(\WW) $ 
(at least for $ \Re s_1, \Re s_2 \geq 0 $, and then apply an analytic continuation).  
Notice that if $ \Re s_1 \geq 0 $ and $ \Re s_2 \geq 0 $, 
then we can take a limit $ v \downarrow 0 $ 
inside the integral, and we get 
$ \Xi_s(\wtilde) = (2 \pi i)^{n s_1 + (n - d) s_2} P_1(\wtilde)^{s_1} P_2(\wtilde)^{s_2} $ 
with an appropriate choice of the branch of exponents.   
In particular, we get 
\begin{equation*}
\Xi_{(0, 0)}(\wtilde) = \delta(\wtilde) 
\qquad \text{(delta distribution supported in the origin)} .
\end{equation*}

\begin{remark}
We will rewrite $ \Xi_s $ more explicitly as a linear sum of the distribution supported in 
the closure of each open orbits $ \calorbit_i $'s 
afterwards.  
See Lemma~\ref{lemma:restriction.hyperfunction.Xi}.
\end{remark}

\begin{theorem}\label{thm:FT.Kernelp}
We get a formula for the Fourier transform of the integral kernel 
$ \Kernelpcone_s $ in the sense of a tempered distribution:
\begin{equation*}
\dfrac{1}{
\Gamma_d(s_1 + \dfrac{d + 1}{2}) \, 
\Gamma_d(s_2 + \dfrac{n}{2}) \, 
\Gamma_{n - d}(s_1 + s_2 + \dfrac{n + 1}{2})
} \, 
\widehat{\Kernelpcone_s} 
= \dfrac{c(s)}{\Gamma_d(- s_2)} \, 
\Xi_{-(s_1 + \frac{d + 1}{2}), -(s_2 + \frac{n}{2})} ,
\end{equation*}
where 
\begin{equation*}
c(s) = (2 \pi)^{\frac{n(n - 1)}{4}} \pi^{- 2 d (s_2 + \frac{n}{4})} ,
\end{equation*}
and $ \Gamma_k(\alpha) $ is given in \eqref{eqn:multiple.Gamma}.
In particular, $ \widehat{\Kernelpcone_s} $ has a pole at 
$ s = -\frac{1}{2} (d+1, n) $ and there the first residue is 
a constant multiple of the delta distribution:
\begin{equation*}
\delta 
= 
\dfrac{1}{c(s)}
\dfrac{\Gamma_d(s_2 + \dfrac{d + 1}{2}) \, \Gamma_d(- s_2)}{\Gamma_{\tildeOmega}(s)} \cdot 
\widehat{\Kernelpcone_s} \Bigm|_{s = - \frac{1}{2} (d+1, n)}  .
\end{equation*}
\end{theorem}

\begin{corollary}
If $ \varphi \in \SchwarzSpace(\WW) $ is supported in the closure of 
the enhanced positive cone $ \tildeOmega $, we get a functional equation:
\begin{equation}
\begin{aligned}
\dfrac{1}{\Gamma_{\tildeOmega}(s)} \,  & \, \Zetapcone(\widehat{\varphi}; s_1, s_2) 
\\
&= \dfrac{c(s) (-2 \pi i)^{- (n s_1 + (n - d) s_2 + \frac{n(n + 1)}{2})}}
{\Gamma_d(s_2 + \dfrac{d + 1}{2}) \, \Gamma_d(- s_2)} \, 
\Zetapcone(\varphi; -(s_1 + \frac{d + 1}{2}), -(s_2 + \frac{n}{2}))
\\
&= \dfrac{c(s)}{(-2 \pi i)^{n s_1 + (n - d) s_2 + \frac{n(n + 1)}{2}} (-\pi)^{d}} 
\prod_{j = 1}^d \sin (s_2 + \frac{d - j}{2}) \; \times
\\
& \hspace*{.4\textwidth}
\Zetapcone(\varphi; -(s_1 + \frac{d + 1}{2}), -(s_2 + \frac{n}{2}))
\end{aligned}
\end{equation}
\end{corollary}

\begin{remark}
For $ d = 1 $, 
Suzuki \cite{Suzuki.1979} proved the analytic continuation 
(Theorem~\ref{thm:analytic.continuation}) and 
the formula for Fourier transform (Theorem~\ref{thm:FT.Kernelp}) above.  

In a general theory of zeta integrals for several variables, 
in an abstract way, 
these formulas are already established 
by F.~Sato (\cite{Sato.PVzeta.I.1982, Sato.PVzeta.II.1983, Sato.PVzeta.III.1982}) 
as we mentioned above, but determining explicit gamma factors/constants or residues 
with the location of poles are different issues.  
Recently the authors were informed that F.~Sato has obtained explicit formulas 
in the present case, or rather, its ``partial dual'' (unpublished).  
We thank Professor F.~Sato for sending us a draft on this subject.
\end{remark}

\section{Proof of the formula of the Fourier transform}

Let us explain how to prove Theorem~\ref{thm:FT.Kernelp}.

For that purpose, 
we suppose a test function is a product of functions in separate variables so that 
$ \varphi(\ztilde) = \varphi_1(z) \varphi_2(y) $.  
By the definition of the Fourier transform of a distribution, we have 
\begin{align}
\int_{\WW} \widehat{\Kernelpcone_s}(\wtilde) 
\varphi(\wtilde) d\wtilde 
&= 
\int_{\WW} {\Kernelpcone_s}(\wtilde) \widehat{\varphi}(\wtilde) d\wtilde 
\notag
\\
&=
\int_{\Omega} (\det z)^{s_1 + s_2} \widehat{\varphi_1}(z) dz 
\int_{\EE} (\det \transpose{y} z^{-1} y)^{s_2} \widehat{\varphi_2}(y) dy .
\label{eqn:FTK.eq1}
\end{align}

\subsection{Fourier transform of the determinant of quadratic maps}

Let us denote the first integral in \eqref{eqn:FTK.eq1} over $ y \in E $ by $ I_1(z) $.  
Now write $ z = \transpose{g} g $ for some $ g \in \LL $.  
Since $ z $ is a positive definite symmetric matrix, this is certainly possible and even 
$ g $ can be taken from the lower triangular Borel subgroup.  
Then, after replacing $ y $ by $ \transpose{g} y $, we get 
\begin{equation*}
I_1(z) = 
\int_E \widehat{\varphi_2}(\transpose{g} y) (\det \transpose{y}y)^{s_2} |\det g|^d \, dy .
\end{equation*}
It is easy to check that 
$ \widehat{\varphi_2}(\transpose{g} y) = |\det g|^{-d} \widehat{L_g \varphi_2}(y) $, 
where 
$ L_g \varphi_2(x) = \varphi_2(g^{-1} x) $ 
is the left translation.  
Thus, if we put $ \psi(x) = L_g \varphi_2(x) $ temporarily, 
\begin{equation*}
I_1(z) = 
\int_E \widehat{\psi}(y) (\det \transpose{y}y)^{s_2} dy , 
\end{equation*}
which is the Fourier transform of 
the complex power of the determinant of the quadratic map $ Q(y) = \transpose{y} y $ 
associated to the representation of the Jordan algebra $ \VV_d = \Sym_d(\R) $.  
Do not confuse this with our Jordan algebra $ \VV = \Sym_n(\R) $.  
For this, we have the following

\begin{lemma}[{\cite[Theorem~2]{Clerc.2002}}; see also \cite{Sato.M.1970.PV}]
Under the setting above and a complex variable $ \alpha \in \C $, we have
\begin{equation*}
\int_E \widehat{\psi}(y) (\det \transpose{y}y)^{\alpha} dy 
= \pi^{- 2 d (\alpha + \frac{n}{4})} \dfrac{\;\Gamma_d(\alpha + \dfrac{n}{2})\;}{\Gamma_d(- \alpha)} 
\int_E \psi(x) (\det \transpose{x} x)^{- \alpha - \frac{n}{2}} \, dx 
\end{equation*}
as tempered distributions.
\end{lemma}

\begin{remark}
Note that our normalization of the Fourier transform is different from Clerc's.  
Also we emphasize that $ Q(y) = \transpose{y} y $ is the quadratic map 
associated to the Jordan algebra $ \VV_d = \Sym_d(\R) $ of smaller size. 

The formula in the above lemma is the zeta integral of PV 
$ (\GL_n(\C), \Mat_{n,d}(\C)) $, and in this sense, 
M.~Sato gave it as an `exercise' (or an example of such formulas) 
in his fundamental paper \cite{Sato.M.1970.PV} on the theory of PV.  
On the other hand, Clerc \cite{Clerc.2002} studied representations of Euclidean Jordan algebras in general, 
and he established a formula for Fourier transform of 
the determinant of quadratic maps associated to the representations in a broader context.
\end{remark}

From this lemma, we continue 
\begin{align*}
I_1(z) &= 
\int_E \widehat{\psi}(y) (\det \transpose{y}y)^{s_2} dy 
\\
&= \pi^{- 2 d (s_2 + \frac{n}{4})}  \dfrac{\Gamma_d(s_2 + \frac{n}{2})}{\Gamma_d(- s_2)} 
\int_E \psi(x) (\det \transpose{x} x)^{- s_2 - \frac{n}{2}} \, dx 
\\
&= (\text{$ \Gamma $-factor}) 
\int_E (L_g \varphi_2)(x) (\det \transpose{x} x)^{- s_2 - \frac{n}{2}} \, dx 
\qquad (\text{replacing $ x $ by $ g x $})
\\
&= (\text{$ \Gamma $-factor}) |\det g|^d 
\int_E \varphi_2(x) (\det \transpose{x} \transpose{g} g x)^{- s_2 - \frac{n}{2}} \, dx 
\\
&= (\text{$ \Gamma $-factor}) (\det z)^{\frac{d}{2}} 
\int_E \varphi_2(x) (\det \transpose{x} z x)^{- s_2 - \frac{n}{2}} \, dx .
\end{align*}
Then the formula \eqref{eqn:FTK.eq1} becomes 
\begin{align}
&
\int_{\WW} \widehat{\Kernelpcone_s}(\wtilde) 
\varphi(\wtilde) d\wtilde 
\notag
\\
&= \pi^{- 2 d (s_2 + \frac{n}{4})}  \dfrac{\Gamma_d(s_2 + \frac{n}{2})}{\Gamma_d(- s_2)} 
\int_{\EE} \varphi_2(x) \, dx 
\int_{\Omega} \widehat{\varphi_1}(z) (\det z)^{s_1 + s_2 + \frac{d}{2}} (\det \transpose{x} z x)^{- s_2 - \frac{n}{2}} dz .
\label{eqn:FTK.eq2}
\end{align}

\subsection{Laplace transform over the positive cone and the gamma constants}

To compute the integral in \eqref{eqn:FTK.eq2} with respect to $ z $, we put a convergence factor 
$ e^{- \trace z v^2} \; (v \in \Omega) $ 
and then take limit $ v \downarrow 0 $.  Thus, we consider an integral below 
($ \alpha := s_1 + s_2 + d/2, \;\; \beta := - s_2 - n /2 $).
\begin{align*}
I_2(v) 
&= 
\int_{\Omega} e^{- \trace z v^2}  
\widehat{\varphi_1}(z) (\det z)^{\alpha} (\det \transpose{x} z x)^{\beta} dz 
\\
&= 
\int_{\Omega} e^{- \trace z v^2}  
\Bigl\{ \int_{V} \varphi_1(w) e^{- 2 \pi i \trace z w} dw \Bigr\} 
(\det z)^{\alpha} (\det \transpose{x} z x)^{\beta} dz 
\\
&= 
\int_{V} \varphi_1(w) dw 
\int_{\Omega} 
e^{- \trace z (v^2 + 2 \pi i w)} 
(\det z)^{\alpha} (\det \transpose{x} z x)^{\beta} dz 
\end{align*}
Since the last integral over $ z $ is analytic in the variable 
$ v^2 + 2 \pi i w \in \Omega + i V \subset \VV_{\C} $, 
we first calculate it when $ w = 0 $, then make an analytic continuation.  
Thus we are to consider 
\begin{align*}
I_2'(v) 
&=
\int_{\Omega} 
e^{- \trace \transpose{v} z v}
(\det z)^{\alpha} (\det \transpose{x} z x)^{\beta} dz .
\\
\intertext{
Substituting $ \transpose{v} z v $ by $ z $, we obtain }
&= (\det v)^{- 2 (\alpha + \frac{n + 1}{2})} 
\int_{\Omega} 
e^{- \trace z}
(\det z)^{\alpha} (\det \transpose{(v^{-1} x)} z (v^{-1} x))^{\beta} dz 
\end{align*}
The following lemma is easy to derive.

\begin{lemma}
For $ x \in E $, put 
\begin{equation*}
\Phi(x) = 
\int_{\Omega} 
e^{- \trace z}
(\det z)^{\alpha} (\det \transpose{x} z x)^{\beta} dz .
\end{equation*}
Then it satisfies 
{\upshape (1)}\ 
$ \Phi(x a) = (\det a)^{2 \beta} \Phi(x) \;\; (a \in \GL_d(\R)) $; and 
{\upshape (2)}\ 
$ \Phi(u x) = \Phi(x) \;\; (u \in \OO_n(\R)) $. 
\end{lemma}

Note that any $ x \in E = \Mat_{n,d}(\R) $ can be decomposed into 
$ x = u \, \II_d \, a \;\; (u \in \OO_n(\R), a \in \GL_d(\R)) $, 
where 
$ \II_d = \vectwo{1_d}{0} \in \EE $.  
Therefore from the lemma, we get 
\begin{equation*}
\Phi(x) = \Phi(u \, \II_d \, a) = (\det a)^{2 \beta} \Phi(\II_d) 
= (\det \transpose{x} x)^{\beta} \Phi(\II_d) .
\end{equation*}
We call $ \gamma(\alpha, \beta) := \Phi(\II_d) $ \emph{gamma constant}.  
Since 
$ \det(\transpose{\II_d} z \II_d) = \Delta_d(z) $ ($ d $-th principal minor), 
we get 
\begin{equation*}
\gamma(\alpha, \beta) 
= 
\int_{\Omega} 
e^{- \trace z}
(\det z)^{\alpha} \Delta_d(z)^{\beta} dz 
= 
\int_{\Omega} 
e^{- \trace z}
\Delta_n(z)^{\alpha} \Delta_d(z)^{\beta} dz , 
\end{equation*}
which can be expressed in terms of Gindikin's gamma function.

\begin{lemma}[{\cite[Theorem~VII.1.1]{Faraut.Koranyi.1994}}]
The gamma constant is given by 
\begin{equation*}
\gamma(\alpha, \beta) = (2 \pi)^{\frac{n(n-1)}{4}} \Gamma_d(\alpha + \beta + \dfrac{n + 1}{2}) 
\Gamma_{n - d}(\alpha + \dfrac{n - d + 1}{2}) .
\end{equation*}
\end{lemma}

Now, summing up all the relevant information, we get 
\begin{align*}
I_2'(v) 
&= (\det v)^{- 2 (\alpha + \frac{n + 1}{2})} (\det \transpose{x} v^{-2} x)^{\beta} 
(2 \pi)^{\frac{n(n-1)}{4}} \Gamma_d(\alpha + \beta + \dfrac{n + 1}{2}) 
\Gamma_{n - d}(\alpha + \dfrac{n - d + 1}{2})
\\
&= 
(2 \pi)^{\frac{n(n-1)}{4}} 
P_1(v^2)^{- (\alpha + \beta + \frac{n + 1}{2})} P_2(v^2, x)^{\beta}
\Gamma_d(\alpha + \beta + \dfrac{n + 1}{2}) 
\Gamma_{n - d}(\alpha + \dfrac{n - d + 1}{2})
\\
\intertext{and}
I_2(v) &= 
(2 \pi)^{\frac{n(n-1)}{4}} 
\Gamma_d(\alpha + \beta + \dfrac{n + 1}{2}) 
\Gamma_{n - d}(\alpha + \dfrac{n - d + 1}{2})
\\
& \qquad \qquad
\int_V \varphi_1(w) 
P_1(v^2 + 2 \pi i w)^{- (\alpha + \beta + \frac{n + 1}{2})} P_2(v^2 + 2 \pi i w, x)^{\beta} dw .
\end{align*}
After substituting $ \alpha = s_1 + s_2 + d/2, \; \beta = - s_2 - n /2 $ and comparing with 
\eqref{eqn:FTK.eq2},
we get Theorem~\ref{thm:FT.Kernelp}.
The proof is completed.

\section{Functional equation for the positive cone}

Here we will rewrite Theorem~\ref{thm:FT.Kernelp} to express the Fourier transform 
as a linear sum of distributions supported in open orbits.  

\subsection{Open orbits and the restriction of the boundary value distribution}\label{subsec:open.orbits}

A classification of open $ G $-orbits is already given in Lemma~\ref{lemma:open.orbits}.  
Thus, an open orbit $ \calorbit \subset \VV \oplus \EE = \WW $ is 
characterized by the signature of $ z $ and that of $ z^{-1} \restrict{}_{[y]} $, 
where $ [y] = \Image y $.
Namely, for quadruples of nonnegative integers $ \rho = (p, q; p', q') $ which satisfies 
$ p + q = n , \; p' + q' = d, \; 0 \leq p' \leq p, \; 0 \leq q' \leq q $, 
the set of all $ \ztilde = (z, y) \in W $ with 
$ \sgn z = (p, q) $ and 
$ \sgn (\transpose{y} z^{-1} y) = (p', q') $ make up an open orbit, 
which we will denote by $ \calorbit_{\rho} $.  
Note that $ \sgn (\transpose{y} z^{-1} y) $ is the same as that of $ z^{-1} \restrict{}_{[y]} $.  

In \S~\ref{section:FT.main.results} we introduced 
the boundary value distribution $ \Xi_s(z, y) $ which is real analytic in $ \ztilde = (z, y) $ 
and meromorphic in $ s = (s_1, s_2) $.  
On the other hand, we have a tempered distribution supported in each open orbit 
$ \calorbit_{\rho} $:
\begin{equation}\label{eqn:distribution.on.open.orbit}
\Kernelrho_s(\ztilde) 
= \begin{cases}
|{\det z}|^{s_1} \, \bigl|{\det \mattwo{z}{y}{\transpose{y}}{0}}\bigr|^{s_2} 
& \ztilde \in \calorbit_{\rho}
\\
0 
&
\text{otherwise.}
\end{cases}
\end{equation}
Note that,  
since $ \tildeOmega = \calorbit_{(n,0;d,0)} $, 
$ \Kernelpcone_s = \Kernelrho_s $ for $ \rho = (n,0;d,0) $.  

\begin{lemma}\label{lemma:restriction.hyperfunction.Xi}
For $ \rho = (p, q; p', q') $, 
put 
$ u_{\rho}(s) = (2 \pi)^{n s_1 + (n - d) s_2} \exp \frac{\pi i}{2} 
\bigl( (p - q) (s_1 + s_2) - (p' - q') s_2 \bigr) $.  
If $ \Re s_1, \Re s_2 > 0 $, 
we have $ \Xi_s \restrict_{\calorbit_{\rho}} = u_{\rho}(s) \Kernelrho_s $ so that 
it holds 
\begin{equation}
\Xi_s = \sum\nolimits_{\rho} u_{\rho}(s) \Kernelrho_s .
\end{equation}
By analytic continuation, the equation is valid for any $ s \in \C^2 $ 
if both sides are interpreted as meromorphic functions.
\end{lemma}

\begin{proof}
The proof is straightforward, but we give it shortly for the convenience of the readers.  
Note that each orbit $ \calorbit_{\rho} $ has a representative 
$ (z_0, y_0) = ( \mattwo{1_p}{0}{0}{-1_q}, \Biggl( \begin{array}{@{}c@{\,}c@{}} 
1_{p'} & 0 \\[-.1ex] 0 & 0 \\[-.3ex] 0 & 1_{q'}
\end{array}
\Biggr) ) $.   
Therefore a general element in the orbit can be written as 
$ (z, y) = (g, h) \cdot (z_0, y_0) = (g z_0 \transpose{g}, g y_0 \transpose{h}) $.  
Thus for $ v \in \Omega $ we compute 
\begin{equation}\label{eqn:compute.hyperfunction.Xi}
\begin{aligned}
\Xi_s(\ztilde) 
&= \lim_{v \downarrow 0} \det(v + 2 \pi i z)^{s_1} 
\Bigl( (-1)^d \det \mattwo{v + 2 \pi i z}{y}{\transpose{y}}{0} \Bigr)^{s_2}
\\
&= \lim_{v \downarrow 0} \det(v + 2 \pi i z)^{s_1 + s_2} 
\det (\transpose{y} (v + 2 \pi i z)^{-1} y)^{s_2}
\\
&= \lim_{v \downarrow 0} \det (g \transpose{g})^{s_1 + s_2} \det (h \transpose{h})^{s_2} 
\det(v_0 + 2 \pi i z_0)^{s_1 + s_2} 
\det (\transpose{y_0} (v_0 + 2 \pi i z_0)^{-1} y_0)^{s_2}, 
\end{aligned}
\end{equation}
where $ v_0 = g^{-1} v \transpose{g}^{-1} $.  
Notice that $ \det g \transpose{g} = |{\det z}| $ and 
$ \det h \transpose{h} =  |{\det \transpose{y} z^{-1} y}| $.  

Since by the assumption 
$ \Re s_1, \Re s_2 > 0 $ holds and 
we can take the limit $ v \downarrow 0$ in any way.  
So we choose $ v_0 = t 1_n \;\; (t > 0) $ as a scalar matrix and take the limit $ t \downarrow 0 $.  
Then the last formula in \eqref{eqn:compute.hyperfunction.Xi} becomes 
\begin{align*}
&= 
|{\det z}|^{s_1 + s_2}
|{\det \transpose{y} z^{-1} y}|^{s_2} 
\lim_{t \downarrow 0}
(t + 2 \pi i)^{p (s_1 + s_2)} (t - 2 \pi i)^{q (s_1 + s_2)} 
(t + 2 \pi i)^{- p' s_2} (t - 2 \pi i)^{- q' s_2} 
\\
&= 
|{\det z}|^{s_1 + s_2}
|{\det \transpose{y} z^{-1} y}|^{s_2} 
(2 \pi)^{n s_1 + (n - d) s_2} 
\exp \frac{\pi i}{2} 
\bigl( (p - q) (s_1 + s_2) - (p' - q') s_2 \bigr)
\end{align*}
which gives $ u_{\rho}(s) $.  
\end{proof}

\subsection{Functional equation for the enhanced positive cone}

Summarizing Theorem~\ref{thm:FT.Kernelp} and Lemma~\ref{lemma:restriction.hyperfunction.Xi}, we get

\begin{theorem}\label{thm:FEq.on.open.orbits}
The Fourier transform of the integral kernel $ \Kernelpcone_s $ 
satisfies the following functional equation:
\begin{equation*}
\widehat{\Kernelpcone_s} 
=
\dfrac{c(s) \; \Gamma_{\tildeOmega}(s)}{\Gamma_d(s_2 + \dfrac{d + 1}{2}) \, \Gamma_d(- s_2)} 
\sum_{\rho} 
u_{\rho}(-s - \tfrac{1}{2}(d + 1, n)) \;
\Kernelrho_{-s - \frac{1}{2}(d + 1, n)} ,
\end{equation*}
where $ \rho =(p, q; p', q') $ moves over all the possible parameters of open orbits and 
$ u_{\rho}(s) $ is given in Lemma~\ref{lemma:restriction.hyperfunction.Xi}.
The rest of the notation is the same as those in Theorem~\ref{thm:FT.Kernelp}.
\end{theorem}

\section*{Acknowledgements}
The first author (K.\,N.)  is supported by JSPS KAKENHI Grant Numbers \#{16K05070}.  
He also thanks for warm hospitality during his visits to 
the Department of Mathematics, Aarhus University.


\section*{References}

\medskip


\skipover{
\def\cftil#1{\ifmmode\setbox7\hbox{$\accent"5E#1$}\else
  \setbox7\hbox{\accent"5E#1}\penalty 10000\relax\fi\raise 1\ht7
  \hbox{\lower1.15ex\hbox to 1\wd7{\hss\accent"7E\hss}}\penalty 10000
  \hskip-1\wd7\penalty 10000\box7} \def\cprime{$'$} \def\cprime{$'$}
  \def\Dbar{\leavevmode\lower.6ex\hbox to 0pt{\hskip-.23ex \accent"16\hss}D}
\providecommand{\newblock}{}

}

\def\cftil#1{\ifmmode\setbox7\hbox{$\accent"5E#1$}\else
  \setbox7\hbox{\accent"5E#1}\penalty 10000\relax\fi\raise 1\ht7
  \hbox{\lower1.15ex\hbox to 1\wd7{\hss\accent"7E\hss}}\penalty 10000
  \hskip-1\wd7\penalty 10000\box7} \def\cprime{$'$} \def\cprime{$'$}
  \def\Dbar{\leavevmode\lower.6ex\hbox to 0pt{\hskip-.23ex \accent"16\hss}D}
\providecommand{\bysame}{\leavevmode\hbox to3em{\hrulefill}\thinspace}
\providecommand{\MR}{\relax\ifhmode\unskip\space\fi MR }
\providecommand{\MRhref}[2]{%
  \href{http://www.ams.org/mathscinet-getitem?mr=#1}{#2}
}
\providecommand{\href}[2]{#2}
\renewcommand{\MR}[1]{}

\end{document}